\documentclass[12pt]{article}
\usepackage{amsfonts,amsthm, amsmath}
\usepackage{epsfig}
\usepackage{makeidx}
\usepackage{graphicx}

\makeindex
\def\E{{\mathbb E}}
\def\P{{\mathbb P}}

\newcommand{\rr}{{\mathbb R}}

\newcommand{\ncom}{\newcommand}
\ncom{\ul}{\underline}
\ncom{\beq}{\begin{equation}}
\ncom{\eeq}{\end{equation}}
\ncom{\bea}{\begin{eqnarray*}}
\ncom{\eea}{\end{eqnarray*}}
\ncom{\beqa}{\begin{eqnarray}}
\ncom{\eeqa}{\end{eqnarray}}
\ncom{\nno}{\nonumber}
\ncom{\non}{\nonumber}
\ncom{\ds}{\displaystyle}
\ncom{\half}{\frac{1}{2}}
\ncom{\mbx}{\makebox{.25cm}}
\ncom{\hs}{\mbox{\hspace{.25cm}}}
\ncom{\rar}{\rightarrow}
\ncom{\Rar}{\Rightarrow}
\ncom{\noin}{\noindent}
\ncom{\bc}{\begin{center}}
\ncom{\ec}{\end{center}}
\ncom{\sz}{\scriptsize}
\ncom{\rf}{\ref}
\ncom{\s}{\sqrt{2}}
\ncom{\sgm}{\sigma}
\ncom{\Sgm}{\Sigma}
\ncom{\psgm}{\sigma^{\prime}}
\ncom{\dt}{\delta}
\ncom{\Dt}{\Delta}
\ncom{\lmd}{\lambda}
\ncom{\Lmd}{\Lambda}
\ncom{\Th}{\Theta}
\ncom{\e}{\eta}
\ncom{\eps}{\epsilon}
\ncom{\pcc}{\stackrel{P}{>}}
\ncom{\lp}{\stackrel{L_{p}}{>}}
\ncom{\dist}{{\rm\,dist}}
\ncom{\sspan}{{\rm\,span}}
\ncom{\re}{{\rm Re\,}}
\ncom{\im}{{\rm Im\,}}
\ncom{\sgn}{{\rm sgn\,}}
\ncom{\ba}{\begin{array}}
\ncom{\ea}{\end{array}}
\ncom{\hone}{\mbox{\hspace{1em}}}
\ncom{\htwo}{\mbox{\hspace{2em}}}
\ncom{\hthree}{\mbox{\hspace{3em}}}
\ncom{\hfour}{\mbox{\hspace{4em}}}
\ncom{\vone}{\vskip 2ex}
\ncom{\vtwo}{\vskip 4ex}
\ncom{\vonee}{\vskip 1.5ex}
\ncom{\vthree}{\vskip 6ex}
\ncom{\vfour}{\vspace*{8ex}}
\ncom{\norm}{\|\;\;\|}
\ncom{\integ}[4]{\int_{#1}^{#2}\,{#3}\,d{#4}}
\ncom{\vspan}[1]{{{\rm\,span}\{ #1 \}}}
\ncom{\dm}[1]{ {\displaystyle{#1} } }
\ncom{\ri}[1]{{#1} \index{#1}}

\newtheorem{theorem}{\bf Theorem}[section]
\newtheorem{remark}{\bf Remark}[section]
\newtheorem{proposition}{Proposition}[section]
\newtheorem{lemma}{Lemma}[section]
\newtheorem{corollary}{Corollary}[section]

\newtheoremstyle
    {remarkstyle}
    {}
    {11pt}
    {}
    {}
    {\bfseries}
    {:}
    {     }
    {\thmname{#1} \thmnumber{#2} }

\theoremstyle{remarkstyle}

\begin{document}

\newpage

\begin{center}
{\Large \bf Time-Changed Poisson Processes}\\
\end{center}
\vone
\begin{center}
{\bf  A. Kumar$^a$, Erkan Nane$^{b}$,  and P. Vellaisamy$^{a}$}\\
$^{a}${\it Department of Mathematics,
Indian Institute of Technology Bombay,\\ Mumbai-400076, India.}\\
$^{b}${\it Department of Mathematics and Statistics, Auburn University,\\
Auburn, AL 36849 USA.}\\

\end{center}

\vtwo
\begin{center}
\noindent{\bf Abstract}
\end{center}
We consider time-changed Poisson processes, and derive the
governing difference-differential  equations (DDE) these processes. In particular, we consider the
time-changed Poisson processes where the the time-change is inverse Gaussian, or its hitting time process,  and
discuss the  governing DDE's.
The stable subordinator, inverse stable subordinator and
their iterated versions are also considered as time-changes.  DDE's corresponding to probability
mass functions of these time-changed processes are obtained.
Finally, we obtain a new  governing  partial differential equation for the 
tempered stable  subordinator of index $0<\beta<1,$ when $\beta $ is a
rational number. We then use this result to obtain the governing DDE for the mass function of  Poisson process time-changed by tempered stable subordinator.  Our results extend and  complement  the results in Baeumer et al. \cite{B-M-N} and Beghin et al. \cite{BO-1} in several directions.

\vone \noindent{\it Key words:} Hitting times; inverse
Gaussian
 process; stable processes; time-changed process, subordination; tempered stable processes;  difference-differential equation.
\vtwo

\setcounter{equation}{0}
\section{Introduction}

Recently there has been an increasing  interest  to consider time-changed  stochastic processes that yield solutions of fractional Cauchy problems, or  solutions of  higher order partial differential equations (PDE). The factional Cauchy problems can be used to model  various phenomena in a
wide range of scientific areas including physics, telecommunications, turbulence, image processing, biology, bioengineering,
hydrology and finance, see \cite{B-M-N, BO-1, H-K-U, M-S-limitctrw, M-S}. There is an interesting  connection between continuous time random walks and fractional Cauchy problems, see \cite{M-N-V, M-S-limitctrw, M-S}

It is well known that the Poisson process $N(t)$ with parameter $\lambda>0$ solves the following difference-differential equation (DDE)
\beq
\frac{d}{dt}p_k(t)=-\lambda p_k(t)+\lambda p_{k-1}(t),
\eeq
where $p_k(t) = \P\{N(t)=k\}$ is the probability mass function (pmf). The fractional Poisson process $N^{\beta}(t)$, which is a generalization of the Poisson process, solves the following fractional DDE (see e.g. Laskin (2003))
\beq
\frac{d^{\beta}}{dt^{\beta}}p^{\beta}_k(t) = -\lambda p^{\beta}_k(t)+\lambda p^{\beta}_{k-1}(t),
\eeq
where $p^{\beta}_k(t) = \P\{N^{\beta}(t)=k\}$ and
\begin{equation}\label{CaputoDef}
\frac{\partial^\beta u(t,x)}{\partial
t^\beta}=\frac{1}{\Gamma(1-\beta)}\int_0^t \frac{\partial
u(r,x)}{\partial r}\frac{dr}{(t-r)^\beta}
\end{equation}
for $0<\beta<1$, denotes the Caputo fractional derivative (see e.g. Caputo (1967)). Let $\tilde{u}(s,x)=\int_{0}^{\infty}e^{-st}u(t,x)dt = \mathcal{L}_t(u(t,x))$ be the Laplace transform (LT) of $u(t,x)$ with respect to the variable $t$. Then  Laplace transform is given by
\begin{equation}\label{CaputoLT}
\mathcal{L}_t\left(\frac{\partial^\beta u(t,x)}{\partial
t^\beta}\right)=s^\beta \tilde u(s,x)-s^{\beta-1} u(0,x).
\end{equation}

The inverse Gaussian (IG) process $G(t)$ (or IG subordinator) has
been found useful in financial modeling and is defined by (see
Applebaum, 2009, p.\ 54) \beq\label{ig} G(t) = \inf\{s > 0; B(s) +
\gamma s > \delta t \}, \eeq where $B(t)$ is the standard Brownian
motion. Note that $G(t)\sim$ IG$(\delta t, \gamma)$, the inverse
Gaussian distribution with density \beq\label{ig-density} g(x,t) =
\displaystyle{(2\pi)}^{-1/2} (\delta t) x^{-3/2}e^{\delta\gamma
t-\frac{1}{2}(\frac{\delta^2t^2}{ x} + \gamma^2 x)},~~x>0. \eeq
Also, note that \beq H(t) = \inf\{s\geq 0: G(s)>t\} \eeq denotes
its first hitting time process.

 In this paper, we
consider the Poisson process time-changed by $G(t)$ and  by the
process $H(t)$, and investigate their properties. We consider also
the problem of time-change by stable and tempered stable
process, and derive the underlying PDE's. We only treat the case where the process and the time changes are assumed to be  independent.  Our results extend and complement the results in Baeumer et al. \cite{B-M-N} and Orsingher et al. \cite{BO-1} in several directions.

\setcounter{equation}{0}
\section{ IG and its hitting time process as time-changes}
Let $N(t)$ be the Poisson process and $G(t)$ be the IG subordinator. First we consider the time-changed process $N(G(t))$.
The probability mass function $\hat{p}_k(t) = \P(N(G(t))=k)$ of the time-changed process $N(G(t))$ is obtained by
 the standard conditioning argument as
 \beq
 \begin{split}
 \hat{p}_k(t)&= \int_{0}^{\infty}\frac{e^{-\lambda x}(\lambda x)^k}{k!}{(2\pi)}^{-1/2} (\delta t) x^{-3/2}e^{\delta\gamma
t-\frac{1}{2}(\frac{\delta^2t^2}{ x} + \gamma^2 x)}dx
 \\
 &=\sqrt{\frac{2}{\pi}}(\delta t)e^{\delta\gamma t}\frac{\lambda^k}{k!}\left(\frac{\delta t}{\sqrt{\gamma^2+2\lambda}}\right)^{k-1/2}K_{k-1/2}(\delta t\sqrt{\gamma^2+2\lambda}),
 \end{split}
 \eeq
 where $K_{\nu}(z)$ is the modified Bessel function of third kind with
index $\nu$, defined by (see e.g. Abramowitz and Stegun (1992))
\beq\label{bessel}
K_{\nu}(\omega) = \displaystyle\frac{1}{2}\int_{0}^{\infty}x^{\nu-1}e^{-\frac{1}{2}\omega(x+x^{-1})}dx, ~\omega>0.
\eeq
Since, as $t\rightarrow\infty$, $\frac{N(t)}{t}\rightarrow\frac{1}{\lambda}$, a.s. and  $\frac{G(t)}{t}\rightarrow\frac{\delta}{\gamma}$, a.s. (see e.g. Bertoin (1996), p. 92),
we have
\beq
\lim_{t\rightarrow\infty}\frac{N(G(t))}{t}\rightarrow\frac{\delta}{\lambda\gamma},~~\mbox{a.s.}
\eeq
or equivalently $N(G(t))\sim\frac{\delta}{\lambda\gamma}t$, a.s.
This shows that the subordinated process does not explode in any finite interval of time. \\

 Using the result (see J\o rgenson (1992))
 \bea
\E(G(t))^q = \sqrt{\frac{2}{\pi}}\delta(\frac{\delta }{\gamma})^{q-1/2}t^{q+1/2}e^{\delta\gamma t}K_{q-1/2}(\delta\gamma t),
\eea
we easily obtain
\bea
\E N(G(t)) = \frac{\lambda\delta t}{\gamma}\,\, and\,\,
\mbox{Var}(N(G(t))) = \frac{\lambda\delta t}{\gamma}+\lambda^2\sqrt{\frac{2}{\pi}}(\delta t)(\frac{\delta t}{\gamma})^{3/2}e^{\delta\gamma t}K_{3/2}(\delta\gamma t)- (\frac{\lambda\delta t}{\gamma})^2.
\eea

\begin{remark}
 For the particular case $\delta =1$ and $\gamma=0$,
 \begin{align*}
 \P(N(G(t)) = k) &= \int_{0}^{\infty}\frac{e^{-\lambda x}(\lambda x)^k}{k!}\frac{1}{\sqrt{2\pi x^3}}te^{-t^2/{2x}}dx\\
 &=\int_{0}^{\infty}\frac{e^{-\lambda t^2 y}(\lambda t^2 y)^k}{k!}\frac{1}{\sqrt{2\pi y^3}}e^{-1/{2y}}dx ~~\mbox{(put $x=t^2y$)}\\
 &= \int_{0}^{\infty}\frac{e^{-\lambda t^2 y}(\lambda t^2 y)^k}{k!}f_{Y}(y)dy,
 \end{align*}
 where $f_{Y}(y) = \frac{1}{\sqrt{2\pi y^3}}e^{-1/{2y}},~y>0$.
 Hence $N(G(t)) \stackrel{d} = N(\lambda t^2 Y)$, a mixed Poisson process evaluated at time $t^2$.
\end{remark}
\begin{remark}
 Note that $N(G(t))$ is not a renewal process, since here $G(t)$ is not an inverse subordinator (see Kingman (1964); Grandell (1976)). Indeed, it is the hitting time of $B(s)+\gamma s$, which is not a subordinator.
\end{remark}

\noindent The density function $g(x,t)$ of $G(t)$ solves (see Kumar {\it et al.} (2011))
\beq\label{IGgov}
\frac{\partial^2 }{\partial t^2}g(x,t) - 2\delta\gamma \frac{\partial}{\partial t}g(x,t) =  2 \delta^2 \frac{\partial }{\partial x}g(x,t).
\eeq
We have the following result.
\begin{proposition}
 The pmf $\hat{p}_k(t)$ of the subordinated process $N(G(t))$ solves the following DDE:
\beq\label{pg}
\frac{d^2}{dt^2} \hat{p}_k(t) -2\delta\gamma \frac{d}{dt}\hat{p}_k(t)=2\delta^2\lambda[\hat{p}_k(t)-\hat{p}_{k-1}(t)].
\eeq
\end{proposition}
\begin{proof} We have
\beq
\hat{p}_k(t)= \int_{0}^{\infty}p_k(x)g(x,t)dx.
\eeq
This implies by dominated convergence theorem
\beq
\frac{d}{dt}\hat{p}_k(t)=\int_{0}^{\infty}p_k(x)\frac{\partial}{\partial t}g(x,t)dx\ \ and \ \ \frac{d^2}{dt^2}\hat{p}_k(t)=\int_{0}^{\infty}p_k(x)\frac{\partial^2}{\partial t^2}g(x,t)dx.
\eeq
Using the fact $
\lim_{x\to\infty}g(x,t)=0=\lim_{x\to0}g(x,t)$, we obtain
\begin{align*}
\left(\frac{d^2}{dt^2}-2\delta\gamma\frac{d}{dt}\right)\hat{p}_k(t)&=
\int_{0}^{\infty}p_k(x)\left(\frac{\partial^2}{\partial t^2}-2\delta\gamma\frac{\partial}{\partial t}\right)g(x,t)dx\\
&= 2\delta^2\int_{0}^{\infty}p_k(x)\frac{\partial}{\partial x}g(x,t) dx\\
&= -2\delta^2\int_{0}^{\infty}\frac{d}{dx}p_k(x)g(x,t)dx\\
&= -2\delta^2 \int_{0}^{\infty}[-\lambda p_k(x)+\lambda p_{k-1}(x)]g(x,t)dx\\
&=2\delta^2\lambda[\hat{p}_k(t)-\hat{p}_{k-1}(t)],
\end{align*}
proving the result.
\end{proof}

\noindent Next we consider the subordination of hitting time of the process $G(t)$. The right continuous first hitting time of the process $G(t)$ is  defined by
\beq
H(t) = \inf\{s\geq 0: G(s)>t\}.
\eeq
The process $H(t)$ has monotonically increasing continuous sample paths and it is not a L\'evy process.  The process $N(H(t))$ is a renewal process and the inter arrival times follows a tempered Mittag-Leffler distribution.
Using Theorem 4.1 of Meerschaert {\it et al.} (2010), the process $N(H(t))$ is a renewal process whose iid waiting times $J_n$ satisfy
\beq
\P(J_n>x) = E(e^{-\lambda H(x)}),
\eeq
and with LT
\beq
\E(e^{-sJ_n}) = \frac{\lambda}{\lambda+\delta(\sqrt{\gamma^2+2s}-\gamma)}.
\eeq
For $\delta = 1/\sqrt{2}$, the rhs of the above equation reduce to
\beq
\frac{\lambda}{\lambda+ (s+a)^{\beta} - a^{\beta}},
\eeq
where $\beta=1/2$ and $a=\gamma^2/2$.
Also, the density of $J_n$ is
\beq
f_{J_n}(x) = g(x)e^{-ax}\frac{\eta+a^{\beta}}{\eta},
\eeq
with
\beq
g(x) = \frac{d}{dx}[1-E_{\beta}(-\eta x^{\beta})] ~~\mbox{and}~~\eta = \lambda -a^{\beta}
\eeq
(see Example 5.7 of Meerschaert {\it et al.} (2011)).

\begin{remark}
The condition (5.3) given in Meerschaert {\it et. al.}(2010) is satisfied by the process $G(t)$, since
\begin{align*}
\int_{0}^{1}x|ln x|\pi(x) dx &= \int_{0}^{1}x|ln x|\frac{\delta}{\sqrt{2\pi x^3}}e^{-\gamma^2x/2}dx\\
&\leq \frac{\delta}{\sqrt{2\pi}}\int_{0}^{\infty}y^{-1/2}|ln y|dy\\
&= \frac{\delta}{\sqrt{2\pi}}\int_{0}^{\infty}ze^{-z/2}dz= \frac{4\delta}{\sqrt{2\pi}}<\infty.
\end{align*}
Thus their Theorem 5.2 is applicable on IG density.
\end{remark}
\begin{proposition}
 The pmf $\tilde{p}_k(t) = \P(N(H(t))=k)$satisfies
  \begin{align}\label{eq1.22}
\frac{d}{dt}\tilde{p}_k(t) &= \frac{1}{2\delta^2}\left[(\lambda^2-2\delta\gamma\lambda)\tilde{p}_k(t) +
(-2\lambda^2+2\delta\gamma\lambda)\tilde{p}_{k-1}(t)+\lambda^2\tilde{p}_{k-2}(t)\right]-\delta_0(t)\tilde{p}_k(0)\nonumber\\
&\hspace{2cm}+ \frac{1}{2\delta^2}h(0,t)p_k'(0),
\end{align}
where $p'_k(0)= -\lambda, \lambda$ and $0$ for $k=0, 1$ and $k\geq 2$ respectively.
\end{proposition}
\begin{proof} The density function $h(x,t)$ of $H(t)$ satisfies the following PDE (Vellaisamy and Kumar (2011))
\beq\label{h-eq-second-order}
\frac{\partial^2}{\partial x^2}h(x,t)-2\delta\gamma\frac{\partial}{\partial x}h(x,t) = 2\delta^2\frac{\partial}{\partial t}h(x,t)+2\delta^2h(x,0)\delta_0(t).
\eeq

\noindent We have  for $ k\geq 0$
\begin{align*}
\frac{d}{dt}\tilde{p}_k(t)&= \int_{0}^{\infty}p_k(x)\frac{\partial}{\partial t}h(x,t)dx\\
&= \frac{1}{2\delta^2}\int_{0}^{\infty}p_k(x)\left[\frac{\partial^2}{\partial x^2}h(x,t)-2\delta\gamma\frac{\partial}{\partial x}h(x,t)\right]dx-\delta_0(t)\int_{0}^{\infty}p_k(x)h(x,0)dx
\end{align*}

\noindent Now

\beq
\begin{split}
\int_{0}^{\infty}p_k(x)\frac{\partial^2}{\partial x^2}h(x,t)dx&=p_k(x)\frac{\partial}{\partial x}h(x,t)\Big|_0^\infty-\int_{0}^{\infty}p'_k(x)\frac{\partial}{\partial x}h(x,t)dx\\
&=p_k(x)\frac{\partial}{\partial x}h(x,t)\Big|_0^\infty -h(x,t)p'_k(x)\Big|_0^\infty+\int_{0}^{\infty}p^{(2)}_k(x)h(x,t)dx\\
&= -2\delta\gamma p_k(0)h(0,t) + h(0,t)p'_k0x)+\int_{0}^{\infty}p^{(2)}_k(x)h(x,t)dx,
\end{split}
\eeq
since $\lim_{x\rightarrow\infty}h_x(x,t) =0=\lim_{x\rightarrow\infty}h(x,t)$ (see Vellaisamy and Kumar (2011)).
Also
\begin{align}
\int_{0}^{\infty}p_k(x)\frac{\partial}{\partial x}h(x,t)dx &=p_k(x)h(x,t)|_0^\infty-\int_{0}^{\infty}p'_k(x)h(x,t)dx\nonumber\\
&= -p_k(0)h(0,t)-\int_{0}^{\infty}p'_k(x)h(x,t)dx
\end{align}
Using the fact $h_x(0,t) = 2\delta\gamma h(0,t)$, we get
\begin{align}\label{hitting}
\frac{d}{dt}\tilde{p}_k(t) &= \frac{1}{2\delta^2}\left[\int_{0}^{\infty}p^{(2)}_k(x)h(x,t)dx+ 2\delta\gamma \int_{0}^{\infty} p'_k(x)h(x,t)dx+h(0,t)p'_k(0)\right]\nonumber\\
&\hspace{2cm}-\tilde{p}_k(0)\delta_0(t).
\end{align}
The result now follows by substituting
$p^{(2)}_k(x) = \lambda^2[p_k(x)-2p_{k-1}(x)+p_{k-2}(x)]$
in \eqref{hitting}.
\end{proof}

\setcounter{equation}{0}
\section{Stable and inverse stable subordinators as time changes}
Let $f(x,t)$ be the density of a $\beta$-stable subordinator
$D(t)$ with index $0<\beta<1$ . Then the Laplace transform of
$D(t)$ is given by
$$
\E(e^{-sD(t)})=\int_0^\infty e^{-sx}f(x,t)dx=e^{-ts^\beta},
$$
where $s^{\beta}$ is called the Laplace exponent. The  density
$f(x,1)$ of $D(1)$ is infinitely differentiable on $(0,\infty)$,
with the asymptotics as follows: (see Uchaikin and Zolotarev
(1999))
\begin{equation}\label{D-asymptotic-0}
f(x,1)\sim
\frac{(\frac{\beta}{x})^{\frac{2-\beta}{2(1-\beta)}}}{\sqrt{2\pi\beta(1-\beta)}}
e^{-(1-\beta)(\frac{x}{\beta})^{-\frac{\beta}{1-\beta}}}, \ \
\mathrm{as}\ \ x\to 0;
\end{equation}
\begin{equation}\label{D-asymptotic-large}
f(x,1)\sim \frac{\beta}{\Gamma(1-\beta)x^{1+\beta}}, \ \
\mathrm{as} \ \ x\to \infty.
\end{equation}
Note that from \eqref{D-asymptotic-0} and
\eqref{D-asymptotic-large}, we have
\begin{equation}\label{eq3.3n}
\lim_{x\rightarrow 0} f(x,1)= f(0,1)=0~~and~~\lim_{x\rightarrow
\infty} f(x,1)= f(\infty,1)=0.
\end{equation}

\noindent Let $D^*(t)= \inf\{s\geq 0: B(s)>t/\sqrt{2}\}$ denote the first time when the  Brownian motion (with variance $2t$) hits the barrier $t$. The governing equation corresponding to the density function $g(x,t)$ of $D^*(t)$ can be obtained by putting $\delta=1/\sqrt{2}$ and
$\gamma =0$ in the equation \eqref{IGgov}. In this case, $D^*(t)$ is a stable subordinator of index $\frac12$ with Laplace exponent $s^{1/2}$. We have the following proposition.
\begin{proposition}\label{prop-2} Let $D_1^*, D_2^*, \cdots, D_n^*$ be independent $1/2$-stable subordinators.
Then for the iterated composition defined by $D_*^{(n)}(t)=D^*_1oD^*_2o\cdots oD^*_n(t)$, the pmf $\tilde{p}_k(t)= P(N(D_*^{(n)}(t))=k)$,
satisfies the following equation
\beq
\frac{d^{2^n}}{dt^{2^n}}\tilde{p}_k(t) = \lambda [\tilde{p}_k(t)-\tilde{p}_{k-1}(t)].
\eeq
\end{proposition}
\begin{proof} We prove this by induction. The case $n=1$ is obtained  after putting $\delta=1/\sqrt{2}$ and $\gamma=0$
in \eqref{pg}. Let $k(x,t)$ denote the density of $D_1^*(t)$,
which can also be obtained by putting $\delta=1/\sqrt{2}$ and
$\gamma=0$ in \eqref{ig-density}. For the case $n=2$, \beq
\tilde{p}_k(t)=
\int_{0}^{\infty}\int_{0}^{\infty}{p}_k(x_1)k(x_1,x_2)k(x_2,t)dx_1dx_2.
\eeq Integrating by parts and then using \eqref{D-asymptotic-0}
and \eqref{D-asymptotic-large}, we get
\begin{align*}
\frac{d^4}{dt^4}\tilde{p}_k(t)&= \int_{0}^{\infty}\int_{0}^{\infty}{p}_k(x_1)k(x_1,x_2)\frac{\partial^4}{\partial t^4}k(x_2,t)dx_1dx_2\\
&= \int_{0}^{\infty}\int_{0}^{\infty}{p}_k(x_1) \frac{\partial^2}{\partial x_2^2}k(x_1,x_2)k(x_2,t)dx_1dx_2\\
&= \int_{0}^{\infty}\int_{0}^{\infty}{p}_k(x_1) \frac{\partial}{\partial x_1}k(x_1,x_2)k(x_2,t)dx_1dx_2\\
&= - \int_{0}^{\infty}\int_{0}^{\infty}\frac{d}{dx_1}{p}_k(x_1) k(x_1,x_2)k(x_2,t)dx_1dx_2\\
&=  \lambda [\tilde{p}_k(t)-\tilde{p}_{k-1}(t)].
\end{align*}

\noindent The general case also follows in similar way.
\end{proof}

\begin{remark} (i)
The composition of two stable subordinators is a stable subordinator.  Let $D_1, D_2$ be two independent stable subordinators with Laplace exponents $c_is^{\beta_i}$, $i=1,2$.
Then
$$
\E[e^{-sD_1(D_2(t))}]=E[e^{-c_1s^{\beta_1}D_2(t)}]=e^{-tc_2c_1^{\beta_2}s^{\beta_1\beta_2}},
$$
and hence $D_1(D_2(t))$ is a stable subordinator of index $\beta_1\beta_2$. This implies that the composition $E_1(E_2(t))$ is the inverse stable subordinator of index $\beta_1\beta_2$.
Hence, by induction, we get

 $$E_1^{1/2}(E_2^{1/2}(E_3^{1/2}(\cdots E_k^{1/2}(t))))=E^{1/2^k}(t).
$$
\end{remark}


\noindent (ii) By Remark 2.1, we see that $D_*^{(n)}(t)$ is a stable subordinator of index $1/2^n$ and
\begin{equation}\label{composing-n-s-sub}
\E(e^{-sD^{(n)}(t)})=\exp(-ts^{1/2^n}).
\end{equation}

Note also also that $N(D^{(n)}(t))$ is also a L\'evy Process, as $N(t)$ is Poisson process and $D^{(n)}(t)$ is a stable subordinator (by Theorem 30.1 in Sato (1999)) and it has Laplace transform
$$
\E(e^{-sN(D^{(n)}(t))})=e^{-t(\log(\tilde{\mu}(s)))^{1/2^n}}=\exp(-t(\lambda(e^{-s}-1))^{1/2^n}), \ s\geq 0
$$
where $\tilde{\mu}(s)=\E(e^{-sN(1)})=\exp(\lambda(e^{-s}-1))$. By Theorem 30.1 in Sato (1999), the Fourier exponent of $N(D^{(n)}(t))$ can be written as
$$
\E(e^{izN(D^{(n)}(t))})=\exp(-t(\lambda(e^{iz}-1))^{1/2^n}), \ z\in \rr.
$$

\begin{remark}
Let $0<\beta=\frac km <1$ for $k,m$ relatively prime integers,
and let $D(t)$ be a stable subordinator of index $\beta$ with  $\E(e^{-sD(t)})=e^{-cts^{k/m}}$.
In this case, the density $f(x,t)$ of $D(t)$ is a solution of
\begin{equation}\label{pde-stable-sub}
\frac{\partial^{m}}{\partial t^{m}}f(x,t)
=(-c)^{m}\frac{\partial^k}{\partial x^k}f(x,t), \ \ x, t>0,
\end{equation}
see Lemma 3.1 in DeBlassie (2004). Hence, the density $k^*(x,t)$
of $D_*^{(n)}(t)$ satisfies
\begin{equation}\label{pde-stable-sub-Gn}
\frac{\partial^{2^n}}{\partial t^{2^n}}k^*(x,t)
=\frac{\partial}{\partial x}k^*(x,t), \ \ x, t>0,
\end{equation}

\noindent {\bf Alternative proof of Proposition \ref{prop-2}}
Using \eqref{pde-stable-sub-Gn}  one can obtain an  equation for
$\tilde{p}_k(t)=\P(N(D_*^{(n)}(t))=k)$ as follows:
\begin{equation}\begin{split}
\frac{d^{2^n}}{dt^{2^n}}\tilde{p}_k(t)&= \int_{0}^{\infty}p_k(x)\frac{\partial^{2^n}}{\partial t^{2^n}}k^*(x,t)dx\\
&=\int_{0}^{\infty}p_k(x)\frac{\partial}{\partial x}k^*(x,t)dx\\
&=p_k(x)k^*(x,t)|_0^\infty-\int_{0}^{\infty}k^*(x,t)\frac{\partial}{\partial x}p_k(x)dx \\
&=\lambda [\tilde{p}_k(t)-\tilde{p}_{k-1}(t)],
\end{split}
\end{equation}
which follows by integration by parts and using
\eqref{D-asymptotic-0} and \eqref{D-asymptotic-large}.


\end{remark}
\begin{remark}
\noindent Consider the standard Cauchy process $C(t)$ with density
function \beq q(x,t) = \frac{t}{\pi(x^2+t^2)}, ~x\in \mathbb{R}.
\eeq The Cauchy process is a symmetric $\beta$-stable process with
index $\beta=1$.  Its Fourier transform (FT)
\begin{align*}
\E(e^{-iuC(t)})=\hat{q}(u,t) = e^{-t|u|}~~\mbox{or}
~\frac{\partial}{\partial t}\hat{q}(u,t) = -|u| \hat{q}(u,t).
\end{align*}
This implies
\begin{align*}
\frac{\partial^2}{\partial t^2}\hat{q}(u,t) = |u|^2\hat{q}(u,t) =
-(iu)^2\hat{q}(u,t).
\end{align*}
Invert the FT to get \beq \frac{\partial^2}{\partial t^2}q(x,t)
=-\frac{\partial^2}{\partial x^2}q(x,t), \eeq
 since $(iu)^2\hat{q}$ is the FT of $\partial^2q/\partial x^2$.
 The pmf $q_k(t) =
P(N(|C(t)|)=k)$, $k\geq 0$, $t>0$, satisfies the equation
 \beq\label{Cauchy-sub}
 \begin{split}
\frac{d^2}{dt^2}q_k(t)& = -\lambda^2(1-\bigtriangledown)^2 q_k(t) -
\frac{2}{\pi t}\frac{d}{dx}p_k(x)|_{x=0}\\
&=-\lambda^2(1-\bigtriangledown)^2 q_k(t)+
\frac{2}{\pi t}\lambda(1-\bigtriangledown) p_k(x)|_{x=0},
\end{split}
 \eeq
where $\bigtriangledown f_k(t)=f_{k-1}(t)$, the backward shift
operator.

\noindent  Let $f$ be in the domain of the generator, $-\lambda
(1-\bigtriangledown)$, of the Poisson process, so that
$$-\lambda
(1-\bigtriangledown)f(k)=-\lambda(f(k)-f(k-1)).$$ Then
$u(t,k)=\E_k(f(N(|C(t)|)))$ solves

\beq\label{Cauchy-sub-f}
 \begin{split}
\frac{d^2}{dt^2}u(k,t)& = -\lambda^2(1-\bigtriangledown)^2 u(k,t) -
\frac{2}{\pi t}\frac{d}{dx}u(k,x)|_{x=0}\\
&=-\lambda^2(1-\bigtriangledown)^2 u(k,t)+
\frac{2}{\pi t}\lambda(1-\bigtriangledown) f(k),\\
u(0,k)&=f(k).
\end{split}
 \eeq
A general result that includes Equation \eqref{Cauchy-sub-f} as a special case was first proved in Nane \cite{nane3}
\end{remark}

\noindent We next look at the subordination to inverse stable
subordinators.
 The inverse stable subordinator $E(t)$ ( of  $D(t)$ of index
$0<\beta<1$) has density
\begin{equation}\label{E-density}
\begin{split}
m(x,t)&=\frac{\partial}{\partial x}P(E(t)\leq x)=\frac{\partial}{\partial x}(1-P(D(x)\leq t))\\
&=-\frac{\partial}{\partial x}\int_0^{\frac{t}{x^{1/\beta}}}f(u, 1)du=(t/\beta)f(tx^{-1/\beta }, 1)x^{-1-1/\beta}
\end{split}\end{equation}
using the scaling property of the density
$f(x, t)=(t/\beta)f(tx^{-1/\beta }, 1)x^{-1-1/\beta}$ (see Bertoin
(1996)).

\noindent It is clear from \eqref{E-density} that the density
$m(x,t)$ is in $C^\infty((0,\infty)\times(0,\infty))$. Some
additional properties of $m(x,t)$ are given in the following
Lemma whose proof follows from \eqref{D-asymptotic-0},
\eqref{D-asymptotic-large}, \eqref{E-density}, and by taking
Laplace transforms.

\begin{lemma} [Lemma 2.1, Hahn et al (2010)] 
The density  $m(x,t)$ of $E(t)$ satisfies
\begin{description}
\item[(a)] $\lim_{t\to +0}m(x,t)=\delta_0(x)$ in the sense of the
topology of the space of tempered distribution
$\mathcal{D}'(\rr)$; \item[(b)] $\lim_{x\to
+0}m(x,t)=\frac{t^{-\beta}}{\Gamma(1-\beta)}$, $t>0$; \item[(c)]
$\lim_{x\to \infty}m(x,t)=0$, $t>0$; \item[(d)]$\int_0^\infty
e^{-st}m(x,t)dt=s^{\beta-1}e^{-xs^\beta}$.
\end{description}
\end{lemma}

\begin{lemma}[Lemma 2.2,  Hahn et al (2010); Theorem 4.1, Meerschaert and Scheffler (2008)]
The density $m(x,t)$ is the fundamental solution of
\begin{equation}\label{density-pde}
\frac{\partial^\beta m(x,t)}{\partial t^\beta}=-\frac{\partial
m(x,t)}{\partial x}-\frac{t^{-\beta}}{\Gamma (1-\beta)}\delta_0(l)
\end{equation}
in the sense of tempered distributions.
\end{lemma}

\noindent Next, we consider the iterated composition of $H(t)$ for
the case $\delta=1/\sqrt{2}$ and $\gamma=0$. Define $H_*^n(t) =
H^*_1o\cdots o H^*_n(t)$, where $H^*_1(t), \cdots, H_n^*(t)$ are
$n$ independent copies of $H(t)$. Let $q_k(t) = P(N(H_*^n(t))=k)$.
By Equation \eqref{composing-n-s-sub}, $H_*^n(t)$ is the inverse
of $D_*^{(n)}(t)$ with Laplace exponent $s^{1/2^n}$ .

\begin{remark}
Let $E(t)$ be the inverse of a stable subordinator of index $1/m$
then the density $m(x, t)$ of $E(t)$ satisfies
\begin{equation}\label{Et-first-time-derivative}
\begin{split}
\frac{\partial m(x, t)}{\partial t}&=(-1)^m \frac{\partial^m m(x, t)}{\partial x^m},\ \ x, t>0;\\
\frac{\partial^k m(x,t)}{\partial x^k}|_{x=0}&=t^{-(k+1)/m}\frac{(-1)^k}{\Gamma (1-\frac{(k+1)}{m})},\ \ t>0,  k= 0,1,2,\cdots, (m-2); \\
\frac{\partial^{m-1} m(x,t)}{\partial x^{m-1}}|_{x=0}&=0,\ \ t>0;\\
\lim _{x\to\infty}\frac{\partial^k m(x, t)}{\partial x^k}&=0, \ \
t>0, k=0,1,2,\cdots, (m-1).
\end{split}
\end{equation}
The PDE \eqref{Et-first-time-derivative} was obtained by Keyantuo and Lizama \cite{K-L}.
\end{remark}

\noindent Using the last remark we can easily show

\begin{theorem}\label{N-E-1-m}
Let $E(t)$  be inverse stable subordinator of index
$0<\beta=1/m<1,$ for $m=2,3,4, \cdots$. Then the pmf
$q_k(t)=\P(N(E(t))=k)$ solves the following
DDE
 \beq
 \displaystyle\frac{d}{dt}q_k(t) = (-\lambda)^{m}(1-\bigtriangledown)^{m}q_k(t)+\sum_{j=1}^{m-1}\bigg((-\lambda)^j(1-\bigtriangledown)^{j}p_k(0)\bigg)\frac{t^{-(m-j)/m}}{\Gamma (1-\frac{(m-j)}{m})}.
 \eeq
 \begin{proof}[Proof of Theorem \ref{N-E-1-m}]
This follows by using integration by parts and
\eqref{Et-first-time-derivative} above
\begin{equation}
\begin{split}
\frac{d}{dt}q_k(t) &= \int_{0}^{\infty}p_k(x)\frac{\partial}{\partial t}m(x,t)dx\\
& = (-1)^m\int_{0}^{\infty}p_k(x)\frac{\partial^m}{\partial x^m}m(x,t)dx\\
&=(-1)^m\sum_{j=1}^{m}(-1)^j\frac{\partial^{j-1}}{\partial x^{j-1}}p_k(x)\frac{\partial^{m-j}}{\partial x^{m-j}}m(x,t)
\bigg|_{x=0}\ +\int_{0}^{\infty}\frac{\partial^{m}}{\partial x^{m}}p_k(x)m(x,t)dx\\
&=\sum_{j=1}^{m}(-1)^j\frac{\partial^{j-1}}{\partial x^{j-1}}p_k(x)\bigg|_{x=0}t^{-(1+m-j)/m}\frac{(-1)^j}
{\Gamma (1-\frac{(1+m-j)}{m})}\ +\int_{0}^{\infty}\frac{\partial^{m}}{\partial x^{m}}p_k(x)m(x,t)dx\\
&= \sum_{j=1}^{m}\frac{\partial^{j-1}}{\partial x^{j-1}}p_k(x)\bigg|_{x=0}\frac{t^{-(1+m-j)/m}}{\Gamma
(1-\frac{(1+m-j)}{m})}+\int_{0}^{\infty}(-\lambda)^{m}(1-\bigtriangledown)^{m}p_k(x)m(x,t)dx\\
&=\sum_{j=2}^{m}\frac{\partial^{j-1}}{\partial x^{j-1}}p_k(x)\bigg|_{x=0}\frac{t^{-(1+m-j)/m}}
{\Gamma (1-\frac{(1+m-j)}{m})}+(-\lambda)^{m}(1-\bigtriangledown)^{m}q_k(t)\\
&=\sum_{j=1}^{m-1}\frac{\partial^{j}}{\partial
x^{j}}p_k(x)\bigg|_{x=0}\frac{t^{-(m-j)/m}}{\Gamma
(1-\frac{(m-j)}{m})}+(-\lambda)^{m}(1-\bigtriangledown)^{m}q_k(t)
\end{split}
\end{equation}

\noindent Note that the terms $\frac{\partial^{j-1}}{\partial
x^{j-1}}p_k(x)|_{x=0}$ might not be zero, in general. For example,
$p_k(0)=1$ for $k=0$, $p_k(0)=0$ for $k\geq 1$. Also, $
\frac{\partial}{\partial x}p_k(x)|_{x=0}=0 $ for $k\geq 2$, and it
is equal to $\lambda$ for $k=1$ and is equal to $-\lambda$ for
$k=0$.
\end{proof}

\noindent  Let $f$ be in the domain of the generator, $-\lambda
(1-\bigtriangledown)$, of the Poisson process so that
$$-\lambda
(1-\bigtriangledown)f(k)=-\lambda(f(k)-f(k-1)).$$
 Then
$u(t,k)=\E_k(f(N(E(t))))$ solves
 \beq\begin{split}
 \displaystyle\frac{d}{dt}u(t,k) &= (-\lambda)^{m}(1-\bigtriangledown)^{m}u(t,k)+\sum_{j=1}^{m-1}\bigg( (-\lambda)^{j}(1-\bigtriangledown)^{j}f(k)\bigg)\frac{t^{-(m-j)/m}}{\Gamma (1-\frac{(m-j)}{m})};\\
 u(0,k)&=f(k), k\geq 1
 \end{split}
 \eeq
 This was first proved by Baeumer et al. (2009) \cite{B-M-N}
\end{theorem}
\begin{corollary}
The pmf $q_k(t)=\P(N(H_*^n(t))=k)$ solves the following
DDE
 \beq
 \displaystyle\frac{d}{dt}q_k(t) = \lambda^{2^n}(1-\bigtriangledown)^{2^n}q_k(t)+\sum_{j=1}^{2^n-1}\bigg((-\lambda)^j(1-\bigtriangledown)^{j}p_k(0)\bigg)
 \frac{t^{-(2^n-j)/2^n}}{\Gamma (1-\frac{(2^n-j)}{2^n})},
 \eeq

\end{corollary}

\begin{remark} Let $E(t)$ be inverse stable subordinator of index $0<\beta<1$.
Then the density  $q_k(t)= \P\big(N(E(t) )=k\big)$ solves (see
Meerschaert et al. (2011)). \beq \displaystyle
\frac{d^{\beta}}{dt^{\beta}}{q}_k(t) =
-\lambda(1-\bigtriangledown){q}_k(t).
 \eeq

\noindent In particular, when $\delta=1/\sqrt{2}$ and $\gamma=0$,
the density $q^*_k(t)= P\big(N(H_*^n(t)) =k\big)$ solves \beq
\displaystyle \frac{d^{1/2^n}}{dt^{1/2^n}}{q}^*_k(t) =
-\lambda(1-\bigtriangledown){q}^*_k(t),~~n\geq1.
 \eeq

 \noindent Consider $\tilde{q}_k(t) = P\big(N(H^n(E(t))) =k\big)$. Since $H^n(E(t))$ is inverse stable subordinator
 of index $\beta/2^n$, $\tilde{q}_k(t)$ solves
\beq \displaystyle
\frac{d^{\beta/2^n}}{dt^{\beta/2^n}}\tilde{q}_k(t) =
-\lambda(1-\bigtriangledown)\tilde{q}_k(t),~~0<\beta<1,~n\geq1.
 \eeq
\end{remark}
 Arguments similar to the ones above lead to the following result.

  \begin{proposition} The pmf $\tilde{q}_k(t) = \P\big(N(H^n(E(t))=k)$ solves
 \beq
\displaystyle \frac{d^{\beta}}{dt^{\beta}}\tilde{q}_k(t) = \lambda^{2^n}(1-\bigtriangledown)^{2^n}
\tilde{q}_k(t)+\sum_{j=1}^{2^n-1} \bigg(\frac{\partial^{j}}{\partial z^{j}}p_k(z)\bigg|_{z=0}\bigg)
 \bigg[\frac{t^{-\beta(2^n-j)/2^n}U((2^n-j)/2^n)}{\Gamma
 (1-\frac{(2^n-j)}{2^n})}\bigg],
 \eeq
 where $U(\gamma)=\E[D(1)^{\gamma\beta}]$ for $0<\gamma<1$.
 \end{proposition}

 \begin{proof}
We can write
 \begin{equation}\tilde{q}_k(t) = \P\big(N(H_*^n(E(t))) =k\big)=\int_0^\infty \P\big(N(H_*^n(x)) =k\big)m(x,t)dx=\int_0^\infty
 {q}^*_k(x)m(x,t)dx.
 \end{equation}
  Using the fact that the density $m(x,t)$ is the fundamental solution of
\begin{equation}\label{density-pde-2}
\frac{\partial^\beta m(x,t)}{\partial t^\beta}=-\frac{\partial
m(x,t)}{\partial x}-\frac{t^{-\beta}} {\Gamma
(1-\beta)}\delta_0(x),
\end{equation}
 in the sense of tempered distributions, we  get
 \begin{equation}
 \begin{split}
 \frac{d^{\beta}}{dt^{\beta}}\tilde{q}_k(t)&=\int_0^\infty {q}^*_k(x)\frac{\partial^\beta m(x,t)}
 {\partial t^\beta}dx\\
 &=\int_0^\infty {q}^*_k(x)\bigg[-\frac{\partial m(x,t)}{\partial x}-\frac{t^{-\beta}}{\Gamma (1-\beta)}\delta_0(x)
 \bigg]dx\\
 &=-\int_0^\infty {q}^*_k(x)\frac{\partial m(x,t)}{\partial x}dx-q_k(0)\frac{t^{-\beta}}{\Gamma (1-\beta)}\\
 &=-q^*_k(0)\frac{t^{-\beta}}{\Gamma (1-\beta)}-{q}^*_k(x)m(x,t)|_{x=0}^\infty+\int_0^\infty \frac{d {q}^*_k(x)}{dx}
 m(x,t)dx\\
 &=-q^*_k(0)\frac{t^{-\beta}}{\Gamma (1-\beta)}+{q}^*_k(0)m(0,t)\\
 & +\int_0^\infty m(x,t)\bigg[\lambda^{2^n}(1-\bigtriangledown)^{2^n}q_k(x)+\sum_{j=1}^{2^n-1}\bigg(\frac{\partial^{j}}{\partial z^{j}}p_k(z)\bigg|_{z=0}\bigg)\frac{x^{-(2^n-j)/2^n}}{\Gamma (1-\frac{(2^n-j)}{2^n})}\bigg]dx\\
 &=-q^*_k(0)\frac{t^{-\beta}}{\Gamma (1-\beta)}+{q}^*_k(0)\frac{t^{-\beta}}{\Gamma (1-\beta)}+\lambda^{2^n}
 (1-\bigtriangledown)^{2^n}\int_0^\infty m(x,t)q_k(x)dx\\
 &+\sum_{j=1}^{2^n-1} \bigg(\frac{\partial^{j}}{\partial z^{j}}p_k(z)\bigg|_{z=0}\bigg)
 \int_0^\infty\bigg[\frac{x^{-(2^n-j)/2^n}}{\Gamma (1-\frac{(2^n-j)}{2^n})}\bigg]m(x,t)dx\\
 &=\lambda^{2^n}(1-\bigtriangledown)^{2^n}\tilde{q}_k(t)+\sum_{j=1}^{2^n-1}
 \bigg(\frac{\partial^{j}}{\partial z^{j}}p_k(z)\bigg|_{z=0}\bigg) \E\bigg[\frac{(E(t))^{-(2^n-j)/2^n}}{\Gamma
 (1-\frac{(2^n-j)}{2^n})}\bigg].\\
 \end{split}
 \end{equation}
  We can calculate the terms $\E[E(t)^{-\gamma}]$ for $0<\gamma<1$ as follows:
  First $E(t)\stackrel{(d)}{=}(D(1)/t)^{-\beta}$ by Corollary 3.1 Meerschaert and Scheffler (2004), hence
  $$
  \E[E(t)^{-\gamma}]=\E[(D(1)/t)^{\gamma\beta}]=t^{-\beta\gamma}\E[D(1)^{\gamma\beta}]=:t^{-\beta\gamma}U(\gamma)<\infty
  $$
   Hence, we have
   \beq
    \frac{d^{\beta}}{dt^{\beta}}\tilde{q}_k(t)=\lambda^{2^n}(1-\bigtriangledown)^{2^n-1}
    \tilde{q}_k(t)+\sum_{j=1}^{2^n-1} \bigg(\frac{\partial^{j}}{\partial z^{j}}p_k(z)\bigg|_{z=0}\bigg)
    \bigg[\frac{t^{-\beta(2^n-j)/2^n}U((2^n-j)/2^n)}{\Gamma
    (1-\frac{(2^n-j)}{2^n})}\bigg].
   \eeq

\end{proof}
\setcounter{equation}{0}
\section{Inverse of tempered stable processes as time-changes}
Note that tempered stable processes are obtained by exponential
tempering in the distribution of stable processes, see Rosinski \cite{Rosinski2007} for more details on tempering stable processes. The advantage
of tempered stable process over stable process is that they are
also infinitely divisible, and they have  moments of all order. Let
$f(x,t)$, $0<\beta<1$ denotes the density of a stable process (stable subordinator) with
LT
 \beq \int_{0}^{\infty}e^{-sx}f(x,t)
dx= e^{-ts^{\beta}}. \eeq A tempered stable subordinator
$D_{\mu}(t)$ has a density

\beq\label{ts-density} f_{\mu}(x,t)= e^{-\mu x+\mu^{\beta}t}
f(x,t),~~ \mu>0. \eeq
 The
L\'evy measure corresponding to a tempered stable process is given
by (see e.g. Cont and Tankov (2004, p. 115))

 \beq \pi_{D_\mu}(x) = \frac{ce^{-\mu
x}}{x^{\beta+1}},~c>0, ~x>0, \eeq which implies $\displaystyle
\int_{0}^{\infty}\pi_{D_\mu}(x)dx = \infty$ and hence using Theorem 21.3
of Sato (1999),
  the sample paths of $D_{\mu}(t)$ are strictly increasing since jumping times are dense in $(0,\infty)$. Further the LT
\beq\label{tempered-LT}
\tilde{f}_{\mu}(s,t)=\int_{0}^{\infty}e^{-sx}f(x,t)dx =
e^{-t((s+\mu)^{\beta}-\mu^{\beta})}. \eeq

\begin{proposition}
 The density function $f_{\mu}(x,t)$ of the tempered stable  process of $D_{\mu}(t)$ satisfies the following PDE.
 For $\beta =\frac{1}{m}, ~m\geq 2$
\beq\label{ts-pde} \sum_{j=1}^{m}(-1)^j {m\choose j}
\mu^{(1-j/m)}\frac{\partial^j}{\partial t^j}f_{\mu}(x,t) =
\frac{\partial}{\partial x}f_{\mu}(x,t). \eeq
\end{proposition}
\begin{proof}
 We prove this result by induction. From \eqref{tempered-LT} we have
 \beq
 \tilde{f}_{\mu}(s,t) = e^{-t((s+\mu)^{\beta}-\mu^{\beta})}.
 \eeq
 For $m=2$,
 \beq\label{ts1}
 \frac{\partial}{\partial t}\tilde{f}_{\mu}(s,t) = -((s+\mu)^{1/2}-\mu^{1/2})\tilde{f}_{\mu}(s,t)
 \eeq
 and
 \beq\label{ts2}
 \frac{\partial^2}{\partial t^2}\tilde{f}_{\mu}(s,t) = ((s+\mu)^{1/2}-\mu^{1/2})^2\tilde{f}_{\mu}(s,t).
 \eeq
 Using \eqref{ts1}, \eqref{ts2} and using the fact that $f_{\mu}(0,t) =0$ (see \eqref{eq3.3n} and \eqref{ts-density}), we get
 \begin{align*}
 \left( \frac{\partial^2}{\partial t^2}-2\mu^{1/2}\frac{\partial}{\partial t}\right)\tilde{f}_{\mu}(s,t)
  =s\tilde{f}_{\mu}(s,t)-f_{\mu}(0,t).
 \end{align*}
Inverting the LT to get \beq \left( \frac{\partial^2}{\partial
t^2}-2\mu^{1/2}\frac{\partial}{\partial
t}\right)\tilde{f}_{\mu}(s,t) = \frac{\partial}{\partial
x}\tilde{f}_{\mu}(s,t).
\eeq

Simililarily, for $m=3$ \beq \left(\frac{\partial^3}{\partial
t^3}-3\mu^{1/3}\frac{\partial^2}{\partial t^2}+
3\mu^{2/3}\frac{\partial}{\partial t}\right)\tilde{f}_{\mu}(s,t) =
(-1)^3\frac{\partial}{\partial x}\tilde{f}_{\mu}(s,t). \eeq

The result also follows similarily for a general $m$.
\end{proof}

\begin{remark}
Let $r_k(t)=P(N(D_{\mu}(t))=k)$. Then the pmf $r_k(t)$ satisfies

\beq \sum_{j=1}^{m}(-1)^j {m\choose j} \mu^{(1-j/m)}\frac{d^j}{d
t^j}r_k(t) = \lambda(1-\bigtriangledown)r_k(t). \eeq

\end{remark}

\noindent Using the steps as the proof of Proposition 4.1., we can
show that
 the density function $m_{\mu}(x,t)$ of the hitting time process $E_{\mu}(t)$ of
$D_{\mu}(t)$ satisfies, for $\beta =\frac{1}{m}$ and $m\geq 2$,

 \beq\label{tempered}
 \sum_{j=1}^{m}(-1)^j {m\choose
j} \mu^{(1-j/m)}\frac{\partial^j}{\partial x^j}m_{\mu}(x,t) =
\frac{\partial}{\partial t}m_{\mu}(x,t)+\delta_0(t)m_{\mu}(x,0).
 \eeq

\noindent When $\mu =0$,  \eqref{tempered} reduces to \beq
(-1)^m\frac{\partial^m}{\partial x^m}m_0(x,t)
=\frac{\partial}{\partial t}m_0(x,t) +\delta_0(t)m_0(x,0) \eeq

\noindent Let $E_{\mu}(t)$ denote the inverse of the tempered
stable subordinator $D_{\mu}(t)$.  By Theorem 3.1 in
Meerschaert and Scheffler (2008) we have the density of
$E_{\mu}(t)$ as
$$
m_{\mu}(x,t)=\int_0^t \pi_{D_\mu} (t-y,\infty)f_{\mu}(y,x) dy.
$$
Since the density $f_{\mu}(y,x)$ of $D_{\mu}(x)$ is infinitely
differentiable, $m_{\mu}(x,t)$ is also infinitely differentiable.

 Then we have the following result.

\begin{proposition} The pmf $\tilde{r}_k(t) = P(N(E_{\mu}(t)) = k)$ satisfies the following
PDE:
\beq
\begin{split}
 \frac{d}{dt}\tilde{r}_k(t) &= \sum_{j=1}^{m}{m\choose j}\mu^{(1-j/m)}(-\lambda(1-\bigtriangledown))^j\tilde{r}_k(t)
 + \sum_{j=1}^{m}\sum_{k=1}^{j}(-1)^{j+k}{m\choose j}\mu^{(1-j/m)}\\
 &\hspace{3cm}\times\frac{\partial^{k-1}}{\partial x{k-1}}p_k(x)\frac{\partial^{j-k}}{\partial
 x{j-k}}m_{\mu}(x,t)\Big|_{x=0}-\delta_0(t)r_{k}(0).
\end{split}
\eeq
\end{proposition}
\begin{proof}
We have
\begin{align*}
\frac{d}{dt}\tilde{r}_k(t)& =
\int_{0}^{\infty}p_k(x)\frac{\partial}{\partial t}m_{\mu}(x,t)dx\\
&=\sum_{j=1}^{m}(-1)^j {m\choose j}
\mu^{(1-j/m)}\int_{0}^{\infty}p_k(x)\frac{\partial^j}{\partial
x^j}m_{\mu}(x,t)dx-\delta_0(t)r_{k}(0)\\
&=\sum_{j=1}^{m}(-1)^j {m\choose j}
\mu^{(1-j/m)}\Big[\sum_{k=1}^{j}(-1)^k\frac{\partial^{k-1}}{\partial
x^{k-1}}p_k(x)\frac{\partial^{j-k}}{\partial
x^{j-k}}m_{\mu}(x,t)\Big|_{x=0}\\
&\hspace{2cm}+(-1)^j\int_{0}^{\infty}\frac{d^j}{dx^j}p_k(x)m_{\mu}(x,t)dx\Big]-\delta_0(t)r_{k}(0)\\
&=\sum_{j=1}^{m}\sum_{k=1}^{j}(-1)^{j+k}{m\choose
j}\mu^{(1-j/m)}\frac{\partial^{k-1}}{\partial
x^{k-1}}p_k(x)\frac{\partial^{j-k}}{\partial
x^{j-k}}m_{\mu}(x,t)\Big|_{x=0}\\
 &\hspace{2cm}+\sum_{j=1}^{m}(-1)^j {m\choose
j}\mu^{(1-j/m)}(-\lambda(1-\bigtriangledown))^jr_k(t)-\delta_0(t)r_{k}(0).
\end{align*}
Hence the result.
\end{proof}

\vtwo
\noindent
{}

\section{Notation}
\begin{eqnarray*}
\mbox{variable}&\,\,\mbox{density}\\
G(t)&\,\,g(x,t)\\
\delta=1/\sqrt{2},\gamma=0\,\,& k(x,t)\\
H(t)&\,\,h(x,t)\\
\delta=1/\sqrt{2},\gamma=0\,\,& l(x,t)\\
\beta- stable \,D(t)&\,\,f(x,t)\\
\beta/2-stable\,Y(t)&\,\,p(x,t)\\
1/2-stable \,n^* &\,\, k^*(x,t)\\
E(t)&\,\, m(x,t)\\
\delta=1/\sqrt{2},\gamma=0\,\,& l(x,t)\\
1/2-inverse \,stable\, n^* &\,\, l^*(x,t)\\
Caucy \,C(t)&\,\, q(x,t)\\
TS \,D_{\mu}(t)&\,\, f_{\mu}(x,t)\\
Hitting \,time\, E_{\mu}(t)&\,\, m_{\mu}(x,t)
\end{eqnarray*}

\end{document}